\documentclass[12pt]{amsart}
\usepackage[margin=2.5cm]{geometry}
\usepackage[colorlinks,citecolor=red,urlcolor=blue,bookmarks=false,hypertexnames=true,backref]{hyperref} 
\usepackage{amssymb}
\usepackage{bbm,accents}  

\newtheorem{thm}{Theorem}[section]
\newtheorem{lem}[thm]{Lemma}

\theoremstyle{definition}

\newcommand{\C}{\mathbb{C}}
\newcommand{\N}{\mathbb{N}}
\newcommand{\R}{\mathbb{R}}

\newcommand{\Hconv}{\ast_{(1)}}
\newcommand{\Vconv}{\ast_{(2)}}

\renewcommand{\Im}{\operatorname{Im}}

\newcommand{\Hn}{\mathbb{H}^\cdim}

\newcommand{\wid}{\mathrm{w}}
\newcommand{\heit}{\mathrm{h}}
\newcommand{\cent}{\mathrm{c}}
\newcommand{\one}{{(1)}}
\newcommand{\two}{{(2)}}
\newcommand{\cdim}{n}
\newcommand{\hdim}{Q}
\newcommand{\rinabla}{\accentset{\leftharpoonup}\nabla}

\newcommand{\supp}{\operatorname{supp}}

\newcommand{\oper}[1]{\mathcal{#1}}

\newcommand{\norm}[1]{\left\| #1 \right\|}

\title{Flag Hardy spaces and partial differential equations}
\author{Michael G Cowling}
\thanks{The first-named author thanks the conference organisers for their invitation to speak at the Conference in Honour of Jill Pipher.  
Both authors were supported by ARC Discovery Grant DP220100285.
Both authors also express their appreciation to Sun Yat-Sen University, Guangzhou, PRC, whose hospitality enabled much of the work discussed here.}
\address{School of Mathematics and Statistics\\ University of New South Wales\\ UNSW Sydney NSW 2052\\ Australia}

\author{Ji Li}
\address{Department of Mathematics\\ Macquarie University\\ Macquarie Park NSW 2113 \\ Australia}

\begin{document}
% this prints title, author etc. info from above

\maketitle

\section{Introduction}

This note arises from a talk by the first-named author given at the Conference in Honour of Jill Pipher, held at Macquarie University (Sydney, Australia), in June 2024. 
We offer again our congratulations to Jill on her remarkable career.

Much of the mathematics here is taken in part from, or is a result of reflection on, the preprint \cite{CCLLO24}, and is therefore at least in part due to our  co\-authors Peng Chen, Ming-Yi Lee and Alessandro Ottazzi.

After a quick introduction to flag geometry on the Heisenberg group, we discuss the problem of identifying flag atoms, which boils down to a question in PDE.
We establish a conjecture of E.M.~Stein for the flag Hardy space on the Heisenberg group.
The key result on which our arguments hinge is due to Baldi, Franchi and Pansu \cite{BFP22, BFP24}.

\section{Flag geometry on the Heisenberg group}

We are interested in the (noncommutative) Heisenberg group $\Hn$, which we identify with 
$\C^{\cdim} \times \R$.
The group multiplication is given by
\[
(w,s) (z,t) = (w+z,s+t + \tfrac{1}{2}\Im(w \cdot \bar z))
\]
for all $(w,s), (z,t) \in \Hn$.
The Heisenberg group admits dilations $\delta_r: (z,t) \mapsto ( rz, r^2 t)$ that are group automorphisms, and its homogeneous dimension $\hdim$ is equal to $2\cdim+2$, that is, $| \delta_r E| = r^\hdim |E|$, where $|E|$ is the measure of a measurable subset $E$ of $\Hn$.
When the details of the multiplication are not important, we may abbreviate elements of $\Hn$ to $g$.
Thus, convolution on $\Hn$ is given by
\[
f \ast_{\one} f' (g') 
= \int_{\Hn} f(g) f'(g^{-1} g') \,dg
= \int_{\Hn} f(g'g) f'(g^{-1}) \,dg
\qquad \forall g' \in \Hn.
\]
If $f$ lives on the centre $G_2 := \{0\} \times \R$ of $\Hn$, we integrate over $G_2$ only:
\[
\begin{aligned}
f \ast_{^\two} f' (g') = \int_{G_2} f(g_2) f'(g_2^{-1} g') \,dg_2
\qquad \forall g' \in \Hn ,\\
f' \ast_{^\two} f (g') = \int_{G_2} f'( g'g_2) f(g_2^{-1}) \,dg_2
\qquad \forall g' \in \Hn.
\end{aligned}
\]

The Lie algebra of left-invariant vector fields on $\Hn$ is spanned by the fields
\[
\mathcal{X}_j = \partial_{x_j} - \tfrac{1}{2} y_{j} \partial_{t},
\qquad
\mathcal{Y}_j = \partial_{y_j} + \tfrac{1}{2} x_{j} \partial_{t},
\qquad\text{and}\quad
\mathcal{T} = \partial_t;
\]
the nontrivial commutators of these are determined by the relations $[ \mathcal{X}_j, \mathcal{Y}_j] =
\mathcal{T}$ for all $j$.
The fields $\mathcal{X}_j$ and $\mathcal{Y}_j$ are called horizontal, while $\mathcal{T}$ is called vertical.
The horizontal (or subriemannian) gradient is then
\[
\nabla_{(1)} = (\oper{X}_1, \dots, \oper{X}_n, \oper{Y}_1, \dots, \oper{Y}_n) ,
\]
and the associated sublaplacian $\Delta_{(1)}$ is defined to be $ - \sum_{j} \left(\mathcal{X}_j^2 + \mathcal{Y}_j^2\right)$ (the sign makes it a positive operator).
To make the notation consistent, we also write
\[
\nabla_{(2)} = \oper{T} 
\qquad\text{and}\qquad
\Delta_{(2)} = -\oper{T} ^2 .
\]
The vertical vector field $\mathcal{T} = \partial_t$ is a commutator of horizontal vector fields, so two horizontal derivatives control one vertical derivative.

Much of what we write may be extended to more general stratified groups, but more complicated notation is needed.

The basic geometrical objects on the Heisenberg group $\Hn$ are the Korányi balls
\[
B^{(1)}(o,r) := \bigl\{ (z,t) \in \C^{n} \times \R : ( |z|^4 + t^2 )^{1/4} < r \bigr\}  ,
\]
and the basic objects in the centre are intervals
\[
B^{(2)}(0,s) := \bigl\{ (0,t)  \in \C \times \R : |t| < s \bigr\} .
\]
The basic objects in flag geometry are ``tubes'' $T$, or $T(g ,r,s)$, which are sets of the form
\[
g B^{(1)}(0,r) B^{(2)}(0,s) .
\]
The tube $T$ has \emph{width} $\wid(T)$ equal to $r$, \emph{height} $\heit(T)$ equal to $r^2 +s$, and centre $\cent(T)$ equal to $g$.
We write $2T$ for the tube 
\[
g B^{(1)}(0,2r) B^{(2)}(0,4s), 
\]
which is a dilated version of $T$.
If $s < r^2$, then $ T(g, r,0) \subset T(g ,r,s) \subset T(g, \sqrt{2}r,0)$, and these three tubes are comparable in size.  
We normally assume that $s \geq r^2$.

In \cite{CCLLO24}, we also use objects that we call shards, which are fractal versions of dyadic rectangles, and have better disjointness and nestedness properties than tubes.
These properties are useful for proving an appropriate version of Journ\'e's lemma (see \cite{Jou86, Pip86}).
The existence of shards depends on the existence of discrete cocompact subgroups in $\Hn$, which do not exist in general stratified groups, and we do not use them here.

\section
{The atomic Hardy space}
Let $T$ be a tube in $\Hn$.
We say that a function $a_T$ is a \emph{particle associated to $T$} if $a_T$ is supported in $2T$ and  certain \emph{cancellation conditions} hold; we define and discuss these below.
Let $U$ be an open set of finite measure $|U|$ in $\Hn$, and $\mathfrak{T}(U)$ be the set of tubes $T$ contained in $U$.
We say that a function $a$ is an \emph{atom associated to $U$} if  $a = \sum_{T \in \mathfrak{T}(U)} a_T$, where, for all choices of sign functions $\sigma: \mathfrak{T}(U) \to \{ \pm 1\}$,
\[
\biggl\| \sum_{T \in \mathfrak{T}(U)} \sigma(T) a_T \biggr\|_{L^2(\Hn)}  \leq |U|^{-1/2}.
\]
This condition implies the usual square condition that is imposed on particles, but in the general context of stratified groups, where analogues of dyadic rectangles do not exist, it is still utilisable, while the square function condition relies on having some kind of disjointness or nestedness.

In particular, a particle $a_T$ is an atom if and only if
\begin{equation}\label{eq:size-condition}
|T|^{1/2} \norm{a_T}_{L^2(\Hn)} \leq 1.
\end{equation}
We say that a function $f$ belongs to the \emph{flag Hardy space $H_F^1(\Hn)$} if and only if we may write
$f = \sum_{j \in \N} \lambda_j a_j$,
where all $a_j$ are atoms and all $\lambda_j$ are in $\C$, and $\sum_{j \in \N} |\lambda_j| < \infty$.
The atomic norm of $f$ is the infimum of all sums $\sum_{j \in \N} |\lambda_j|$ over all such representations of $f$.

The main theorem of \cite{CCLLO24} shows that the Hardy space $H^1_F(\Hn)$ may be characterised in various ways, in particular, by
\begin{itemize}
  \item boundedness of certain singular integrals (in particular, double Riesz transforms);
  \item boundedness of certain maximal functions (in particular, radial and nontangential versions);
  \item boundedness of certain square functions (with discrete or continuous parameters);
  \item boundedness of certain area functions.
\end{itemize}

All these characterisations involve showing that the boundedness of some operator $\oper{A}$ implies the existence of an atomic decomposition of the function (which is arguably the nontrivial part of the result) and showing that the same operator $\oper{A}$ is bounded from $H_F^1(\Hn)$ to $L^1(\Hn)$; this involves working with particles $a_T$ and estimating $\oper{A} a_T$.
The estimates for the particles are put together to obtain estimates for atoms using the flag version of Journé's lemma.  %which was discussed in Section \ref{sec:Journe}.

\section{cancellation}\label{sec:moments}
At a conference in honour of Guido Weiss in 1990, Elias M. Stein discussed various problems connected with Hardy spaces, and suggested that the particles for the flag Hardy space might be the convolution product $a^{\one} *_{\two} a^{\two}$ of an atom $a^\one$ for the Folland--Stein--Christ--Geller Hardy space $H^1(\Hn)$ \cite{FS82, CG84} and an atom $a^\two$ for the classical Hardy space $H^1(\R)$ on $\R$ \cite{CW77}.
There atoms are characterised by size conditions, support conditions and cancellation conditions.

More precisely, take $a^\one$ supported in $B^\one(g,r)$ of mean $0$, 
and $a^\two$ supported in $B^\two(0,h)$ of mean $0$, 
such that 
\[
\norm{ a^\one }_{L^2(\Hn)} \leq |B^\one(g,r)|^{-1/2} 
\qquad\text{and}\qquad
\norm{ a^\two }_{L^2(\R)} \leq |B^\two(0,h)|^{-1/2};
\] 
then $a := a^\one *_{(2)} a^\two$ should be a normalised particle associated to $T := T(g,r,h)$.

We investigate and prove a version of this conjecture.
Deciding whether a function $a$ supported in a tube $T$ admits a factorisation as a convolution product  $a^{\one} *_{\two} a^{\two}$ seems to be hard, so we need to look for alternative descriptions.

Since $a_2$ has mean $0$, we may write $a_2$ as a derivative:  $a_2 = \nabla_{(2)} b_2$, where 
\[
\norm{ b_2 }_{L^2(\R)} \leq h |B^\two(0,h)|^{-1/2} .
\] 
Hence $a_1 \ast_\two a_2$ is also a derivative:  $a_1 \ast_\two a_2 = \nabla_{(2)} (a_1 \ast_\two b_2)$.
Further, $a_1 \ast_\two b_2$ has mean $0$ on $\Hn$ because $a_1$ has mean $0$.

Hence we consider the following potential cancellation condition:
\begin{gather}\label{eq:wrong-cancellation}
a = \nabla_{\two} A \quad \text{where $\supp(A) \subseteq T$ \  and \ $
\int_{\Hn} A(g) \, dg  = 0$};
\end{gather}
since $A$ is the integral of $a$ along vertical lines, it is easy to check that
\[
|T|^{1/2} \norm{A}_{L^2(\Hn)} \lesssim h,
\]
and $A$ satisfies an appropriate size condition.

Now we consider a family of such functions $a_h$, for $h \geq 1$, given by
\[
a_h(z,t) = h^{-1} a(z,t/h), 
\]
where $a(z,t) = a_1(z) a_2(t)$, $a_1 \geq 0$, $\supp(a_1) \subseteq B^\one(o,1)$ and $\int_{\C^n} a_1(z) \,dz \neq 0$, while $a_2 \neq 0$, $\supp(a_2) \subseteq B^\two(0,1)$ and the zeroth and first moments of $a_2$ vanish.
If \eqref{eq:wrong-cancellation} were the ``right'' cancellation condition, then the $a_h$ would be uniformly bounded multiples of atoms. 
We consider the convolution $a_h \Hconv k$ of $a$ with a singular integral kernel $k$, of the form
\[
k(z,t) =\omega(z) /(|z|^4 + |t|^2)^{-(2n+3)/4},
\]
where $\omega$ is smooth, homogeneous of degree $0$ and has mean $0$ on the unit sphere in $\C^n$, so that $k$ is smooth, homogeneous of degree $-(2n+2)$ and has mean $0$ on the unit sphere in $\Hn$. 
We would expect that $\norm{ a_h \ast_{\one} k }_{L^1(\Hn)}$ is uniformly bounded in $h$.
However, this is not so.

Indeed, if $(z',t') \in \supp(a_h)$ and $|z| > 1$, then $|z-z'|$ is bounded below, and
\begin{align*}
&\int_{\R}   a(z', t') \,h \, k (z-z', ht - ht' - \tfrac{1}{2} \Im(z' \cdot \bar z)) \,dt'  \\
&\qquad=  \int_{\R}   a(z', t') \frac{ h \omega( z - z') }{ ( |z-z'|^4 +  | h(t - t') - \tfrac{1}{2} \Im(z' \cdot \bar z)|^2 )^{(2n+2)/4}} \,dt' \\
&\qquad\to   a(z', t) \int_{\R}  \frac{ \omega( z - z') }{ ( |z-z'|^4 +  |t'|^2 )^{(2n+2)/4}} \,dt' \\
&\qquad=   a(z', t)  \frac{  d_n \omega( z - z') }{  |z-z'| ^{2n} } \\
&\qquad=: \tilde k(z)\,,
\end{align*}
say, where $d_n$ is an appropriate nonzero constant $d_n$ and $\tilde k$ is a classical singular integral kernel on $\C^n$.
Hence, as $h \to \infty$,
\begin{align*}
&\iint_{B(0,1)^c \times \R}  |a_h \ast_{\one} \,k (z, t)| \,dt \,dz \\
&\qquad= \iint_{B(0,1)^c\times \R} \left| \iint_{\C^n \times \R}   a(z', t') \,h \, k (z-z', ht - ht' - \tfrac{1}{2} \Im(z' \cdot \bar z)) \,dz \,dt' \right| \,dt \,dz\\
&\qquad\to \iint_{B(0,1)^c\times \R} \left| \int_{\C^n}   a(z', t) \, \tilde k (z-z') \,dz \right| \,dt \,dz\\
&\qquad= \iint_{B(0,1)^c\times \R} \left| \int_{\C^n}   a_1(z') a_2(t) \, \tilde k (z-z') \,dz \right| \,dt \,dz\\
&\qquad= \infty.
\end{align*}

This establishes that  $\norm{ a_h \ast_{\one} k }_{L^1(\Hn)}$ is \emph{not} uniformly bounded in $h$, and so our putative cancellation condition does not work.

However, writing $a$ as $a^\one \Hconv a^\two$, where $a^\one$ and $a^\two$ have mean $0$ and are supported in $B^\one(g,r)$ and $B^\two(0,h)$ does work, 
\emph{provided that} $a^\one = \nabla_\one \cdot A^\one$ and $a^\two = \nabla_\two A^\two$, where
\[
\norm{ A^\one }_{L^2(\Hn)} \lesssim r \norm{ a^\one }_{L^2(\Hn)} 
\qquad\text{and}\qquad
\norm{ A^\two }_{L^2(\R)} \lesssim h \norm{ a^\two }_{L^2(\R)}  ;
\] 
\emph{if this is so}, then
\begin{equation}\label{eq:def-cancellation}
a = \nabla_{(1)}\cdot{} \nabla_{(2)} A,  
\end{equation}
where $\supp(A) \subseteq T$, and precursor estimates hold:
\begin{equation}\label{eq:def-precursor-est}
\norm{ (\nabla_{(1)}\cdot{})^i \nabla_{(2)}^j A }_{L^2(\Hn)} 
\lesssim  r^{1-i} (r^2+h)^{1-j} \norm{a}_{L^2(\Hn)} ,
\end{equation}
when $i, j \in\{0,1\}$.
We consider $a \ast_{\one} k$, where $k$ is a flag singular integral  kernel,  given by
\begin{equation}\label{eq:flag-kernel}
k = \iint_{\R^+ \times \R^+} \varphi^\one_s \Vconv \varphi^\two_t \,\frac{ds}{s} \,\frac{dt}{t} \,,
\end{equation}
where $\varphi^\one_s$ and $\varphi^\two_t$ are normalised dilates of $\varphi^\one$ and $\varphi^\two$ respectively, and $\varphi^\one \in C^\infty(\Hn)$ and $\varphi^\two\in C^\infty(\R)$, $\supp(\varphi^\one) \subseteq B^\one(o,1)$ and $\supp(\varphi^\two) \subseteq B^\two(0,1)$, and $\varphi^\one$ and $\varphi^\two$ have mean $0$.

If $R$ and $S$ are dyadic rectangles and $R \subseteq S$, then we define
\begin{equation}\label{eq:def-rho}
\begin{aligned}
\rho(R,S)
:=  \frac{\wid(R)}{ \wid(S)} 
+  \frac{\heit(R)}{ \heit(S) } \,.
\end{aligned}
\end{equation}

\begin{lem}\label{thm:singular-integrals-bounded-on-atoms}
Suppose that singular integral kernel $k$ is  as in \eqref{eq:flag-kernel} above.
Then for each particle $a$ associated to a tube $T$, and for each tube $S \supset 2T$,
\[
\int_{S^c} | a \Hconv k (g) | \,dg 
\lesssim_{\boldsymbol\varphi} \rho(R,S) |R|^{1/2} \norm{a_T}_{L^2(G)} .
\]
\end{lem}

\begin{proof}
We suppose that \eqref{eq:def-cancellation} and \eqref{eq:def-precursor-est} hold and that $A$ is supported in $T$. 
By translation and dilation invariance, we may and shall suppose that $\cent(T) = o$, $\wid(T) =1$ and $\heit(T) = h > 1$; we may and shall also suppose that $\cent(S) = o$, $\wid(S) =w^* > 1$ and $\heit(S) = h^* > h$.

We partition $\R^+ \times \R^+$ into four regions $\Omega^j$, given  by
\[
\begin{aligned}
\Omega_1 &= \{ (r,s) \in \R^+ \times \R^+ : r \leq 1, s \leq h \} \qquad&
\Omega_2 &= \{ (r,s) \in \R^+ \times \R^+ : r \leq 1, s > h \} \\
\Omega_3 &= \{ (r,s) \in \R^+ \times \R^+ : r > 1, s \leq h\} &
\Omega_4 &= \{ (r,s) \in \R^+ \times \R^+ : r > 1, s > h\},
\end{aligned}
\]
and we write $a \Hconv k = a \Hconv k^1 + a \Hconv k^2 + a \Hconv k^2 + a \Hconv k^4$, where 
\[
k^i := \iint_{\Omega^i} \varphi^\one_s \Vconv \varphi^\two_t \,\frac{ds}{s} \,\frac{dt}{t} \,.
\]
We treat the four summands $a \Hconv k^i$ separately.
The key to our estimation is the fact that
\begin{equation}\label{eq:sing-int--support}
\begin{aligned}
&\supp \bigl( a \Hconv \varphi_{r,s}\bigr)
\subseteq
T(o,1, h)  T(o, r, s) = T(o, r+1, s+h ) .
\end{aligned}
\end{equation}

First, by definition, $\supp (a \Hconv k^1) \subseteq 2S$, and there is nothing to consider; in any case, $a \Hconv k^1 \in L^1(G)$.

Second, to treat $a \Hconv k^2$, we observe that if $a \Hconv \varphi^\one_s \Vconv \varphi^\two_t(g)  \neq 0$, where $g \in S^c$ and $(s,t) \in \Omega^2$, then $h^* \leq  h + s + 1 \leq h+2$, and so
\[
\begin{aligned}
\int_{S^c} | a \Hconv k^2 (g)| \,d  g 
%&= \int_{S^c} | \int_{h}^{\infty}  c^2 \Vconv (\nabla_{\two}\varphi^\two)_t   \,\frac{dt}{t^3} | \,d  g \\
&= \int_{S^c} \left| \int_{h^*/2}^{\infty} \int_{0}^{1} a \Hconv \varphi^\one_s \Vconv \varphi^\two_t  \,\frac{ds}{s} \,\frac{dt}{t} \right| \,d  g .
\end{aligned}
\]

Recall that $a = \nabla_{(2)} (\nabla_{\one} \cdot A)$, and write
\[
\begin{aligned}
a \Hconv k^2
&=  \int_{h^*/2}^{\infty} \int_{0}^{1} (\nabla_{\two} (\nabla_{\one} \cdot A)) \Hconv \varphi^\one_s \Vconv \varphi^\two_t \,\frac{ds}{s} \,\frac{dt}{t} \\
&=  \int_{h^*/2}^{\infty} \left( \int_{0}^{1} (\nabla_{\one} \cdot A) \Hconv \varphi^\one_s \,\frac{ds}{s}  \right) \Vconv (\nabla_{\two}\varphi^\two)_t \,\frac{dt}{t^2} \\
&=  \int_{h^*/2}^{\infty} c^2 \Vconv (\nabla_{\two}\varphi^\two)_t \,\frac{dt}{t^2} \,, 
\end{aligned}
\]
say, where 
\[
c^2 = \int_0^1  (\nabla_{\one} \cdot A) \Hconv \varphi^\one_s \,\frac{ds}{s} \,.
\]
Now $\nabla_{\one} \cdot A$ is supported in $T$, so $c^2$ is supported in $T^*$, say, where $|T^*| \eqsim |T|$.
Moreover,   by H\"older's inequality and one-parameter singular integral estimates (see, e.g., \cite{CCLLO24}),
\[
\norm{ c^2 }_{L^1(\Hn)} 
\lesssim |T|^{1/2}  \norm{ c^2 }_{L^2(\Hn)} 
\lesssim_{\varphi^\one} |T|^{1/2} \norm{  \nabla_{\one} \cdot A }_{L^2(\Hn)} 
\lesssim h |T|^{1/2}  \norm{ a }_{L^2(\Hn)} .
\]
We conclude that
\[
\begin{aligned}
\int_{S^c} | a \Hconv k^2 (g)| \,d  g 
&= \int_{S^c} \left| \int_{h^*/2}^{\infty}  c^2 \Vconv (\nabla_{\two}\varphi^\two)_t(g)  \,\frac{dt}{t^2} \right| \,d  g \\
&\leq \int_{h^*/2}^{\infty}  \int_{\Hn } |  c^2 \Vconv (\nabla_{\two}\varphi^\two)_t (g)| \,d  g \,\frac{dt}{t^2}  \\
&\leq \int_{h^*/2}^{\infty}  \norm{ c^2 }_{L^1(\Hn)} \norm{ (\nabla_{\two}\varphi^\two)_t  }_{L^1(\R)} \,\frac{dt}{t^2}  \\
& \eqsim_{\varphi^\two} \frac{1}{h^*}  \norm{ c^2 }_{L^1(\Hn)}  \\
& \lesssim_{\varphi^\one} \frac{h}{h^*}  |R|^{1/2} \norm{ a }_{L^2(\Hn)}  .
\end{aligned}
\]

Third, to treat $a \Hconv k^3$, we argue similarly.
If $a \Hconv \varphi^\one_s \Vconv \varphi^\two_t(g)  \neq 0$, where $g \in S^c$ and $(s,t) \in \Omega^3$, then $w^* \leq s+1$, whence
\[
\begin{aligned}
\int_{S^c} | a \Hconv k^3 (g)| \,d  g 
&= \int_{S^c} \left| \int_{w^*/2}^{\infty} \int_{h}^{\infty} a \Hconv \varphi^\one_s \Vconv \varphi^\two_t  (g) \,\frac{dt}{t} \,\frac{ds}{s} \right| \,d  g .
\end{aligned}
\]
Now we write
\[
\begin{aligned}
a \Hconv k^3
&=  \int_{w^*/2}^{\infty} \int_{0}^{h} (\nabla_{\one} \cdot (\nabla_{\two} A)) \Vconv \varphi^\two_t \Hconv \varphi^\one_s \,\frac{dt}{t} \,\frac{ds}{s} \\
&=  \int_{w^*/2}^{\infty} c^3 \Hconv (\rinabla_{\one}\varphi^\one)_s \,\frac{ds}{s^2} \,, 
\end{aligned}
\]
say, where 
\[
c^3 = \int_0^h (\nabla_{\two} A) \Vconv \varphi^\two_t \,\frac{dt}{t} \,.
\]
Hence 
\[
\begin{aligned}
\int_{S^c} | a \Hconv k^3 (g)| \,d  g 
&= \int_{S^c} \left| \int_{w^*/2}^{\infty}  c^3 \Hconv (\rinabla_{\one}\varphi^\one)_s(g)  \,\frac{ds}{s^2} \right| \,d  g \\
&\leq \int_{w^*/2}^{\infty}  \int_{\Hn } |  c^3 \Hconv (\rinabla_{\one}\varphi^\one)_s(g) | \,d  g \,\frac{ds}{s^2}  \\
& \lesssim_{\varphi^\one} \frac{1}{w^*}  \norm{ c^3 }_{L^1(\Hn)}  \\
& \lesssim_{\varphi^\two} \frac{1}{w^*}  |R|^{1/2} \norm{ a }_{L^2(\Hn)}  .
\end{aligned}
\]

Last, we treat $a \Hconv k^4$.
If $a \Hconv \varphi^\one_s \Vconv \varphi^\two_t(g)  \neq 0$, where $g \in S^c$ and $(s,t) \in \Omega^4$, then either $s+1 \geq  w^*$, or $s+1 < w^*$ and $t + s^2 + h \geq h^*$, whence
\[
\begin{aligned}
\int_{S^c} | a \Hconv k^4 (g)| \,d  g 
&\leq \int_{S^c} \int_{w^*/2}^{\infty} \int_{h}^{\infty} |  a \Hconv \varphi^\one_s \Vconv \varphi^\two_t(g) |  \,\frac{dt}{t} \,\frac{ds}{s} \,d  g \\
&\qquad+  \int_{S^c} \int_{1}^ {w^*} \int_{h^*}^{\infty} | a \Hconv \varphi^\one_s \Vconv \varphi^\two_t(g) |   \,\frac{dt}{t} \,\frac{ds}{s} \,d  g .
\end{aligned}
\]
Next, 
\[
\begin{aligned}
a \Hconv \varphi^\one_s \Vconv \varphi^\two_t 
&=  \frac{1}{s t} A  \Hconv  (\rinabla_{\one}\varphi^\one)_s \Vconv (\nabla_{\two}\varphi^\two)_t \,.
\end{aligned}
\]
By estimating much as before, we see that
\[
\begin{aligned}
&\int_{S^c} \int_{w^*/2}^{\infty} \int_{h}^{\infty} | a \Hconv \varphi^\one_s \Vconv \varphi^\two_t(g) |  \,\frac{dt}{t} \,\frac{ds}{s} \,d  g \\
&\qquad\leq \int_{w^*/2}^{\infty} \int_{h}^{\infty} \int_{\Hn} | A \Hconv (\rinabla_{\one}\varphi^\one)_s \Vconv (\nabla_{\two}\varphi^\two)_t (g) |  \,d  g \,\frac{dt}{t^2} \,\frac{ds}{s^2}   \\
&\qquad\leq \int_{w^*/2}^{\infty} \int_{h}^{\infty} \norm{ A }_{L^1(\Hn)} \norm{ (\rinabla_{\one}\varphi^\one)_s }_{L^1(\Hn)} \norm{ (\nabla_{\two}\varphi^\two)_t }_{L^1(\R)}  \,\frac{dt}{t^2} \,\frac{ds}{s^2}  \\
&\qquad \lesssim_{\varphi^\one, \varphi^{\two}} \frac{1}{h w^*}  \norm{ A }_{L^1(\Hn)}  \\
&\qquad \lesssim \frac{1}{w^*}  |R|^{1/2} \norm{ a }_{L^2(\Hn)}  ,
\end{aligned}
\]
while
\[
\begin{aligned}
&\int_{S^c} \int_{1}^{w^*} \int_{h^*/2}^{\infty}  | a \Hconv \varphi^\one_s \Vconv \varphi^\two_t (g) | \,\frac{dt}{t} \,\frac{ds}{s} \,d  g \\
&\qquad\leq \int_{1}^{\infty} \int_{h^*/2}^{\infty} \int_{\Hn} | A \Hconv (\nabla_{\one}\varphi^\one)_s \Vconv (\nabla_{\two}\varphi^\two)_t (g) |  \,d  g  \,\frac{dt}{t^2} \,\frac{ds}{s^2} \\
&\qquad\leq \int_{1}^{\infty} \int_{h^*/2}^{\infty} \norm{ A }_{L^1(\Hn)} \norm{ (\nabla_{\one}\varphi^\one)_s }_{L^1(\Hn)} \norm{ (\nabla_{\two}\varphi^\two)_t }_{L^1(\R)}  \,\frac{dt}{t^2} \,\frac{ds}{s^2}  \\
&\qquad \lesssim_{\varphi^\one,\varphi^\two} \frac{1}{h^*}  \norm{ A }_{L^1(\Hn)}  \\
&\qquad \lesssim \frac{h}{h^*}  |R|^{1/2} \norm{ a }_{L^2(\Hn)}  .
\end{aligned}
\]
In conclusion, 
\[
\int_{(S^*)^c} | a \Hconv k (g) | \,dg 
\lesssim_{\boldsymbol\varphi} \rho(R,S) |R|^{1/2} \norm{a}_{L^2(G)},
\]
as required.
\end{proof}

This lemma, coupled with Journé's lemma, then shows that flag singular integral operators are bounded on the flag Hardy space defined using particles satisfying \eqref{eq:def-cancellation} and \eqref{eq:def-precursor-est}.
In particular, the flag Riesz transformations map this space into $L^1(\Hn)$ and so the flag Hardy space defined using particles satisfying \eqref{eq:def-cancellation} and \eqref{eq:def-precursor-est} is contained in the flag Hardy space defined in \cite{CCLLO24}.
Conversely, the flag Hardy space defined in \cite{CCLLO24} is trivially contained in the flag Hardy space defined using particles satisfying \eqref{eq:def-cancellation} and \eqref{eq:def-precursor-est}, and so these Hardy spaces coincide.

The argument of this proof relies on the precursor estimates \eqref{eq:def-precursor-est}, which follow from assumptions about $a^\one$ and $a^\two$.
The estimate for $a^\two$ is straightforward, but that for $a^\one$ is nontrivial.

On $\R^n$, the Poincaré lemma shows that if $0 \leq k < n$ and a smooth $k$-form $\omega$ is supported in $B(x,r)$ and $d\omega = 0$, where $d$ is the exterior derivative, then we may write $\omega = d\Omega$, where $\Omega$ is a smooth $(k-1)$-form supported in $B(x,r)$; further, if a function $a$ has mean $0$ and $\supp(a) \subseteq B(x,r)$, then we may write the $n$-form $a \,dx_1 \wedge \dots \wedge dx_n$  as $d A$, where the $(n-1)$-form $A$ is supported in $B(x,r)$.   
The forms $\Omega$ and $A$ satisfies $L^2$ precursor estimates \cite{MMM08}.

The argument to treat $n$-forms has to be slightly different to that used to treat forms of lower degrees because the condition imposed on $a$ is different.
The second result implies that a function $a$ of mean $0$ is the divergence of a vector field of controlled size.

On $\Hn$, the Rumin complex replaces the de Rham complex.  
Baldi, Franchi and Pansu \cite{BFP22} found precursor estimates for the corresponding problem on stratified Lie groups, except for the forms of maximum degree.
A very recent paper of Baldi, Franchi and Pansu \cite{BFP24} extends these estimates to forms of maximum degree, and  to general stratified groups.
Hence \eqref{eq:def-precursor-est} holds on the Heisenberg group.

In \cite{CCLLO24}, we formulated the theory in terms of particles $a$ with support in a tube $T$ that satisfy 
\begin{equation}\label{eq:weak-cancellation}
a = \Delta_{(1)} \Delta_{(2)} A, 
\end{equation}
where  $\supp(A) \subseteq T$.
For such a particle to be an atom, we add the standard size condition:
\[
|T|^{1/2}  \norm{ a_T }_{L^2(\Hn)} \leq 1 .
\]
Precursor estimates similar to \eqref{eq:def-precursor-est} are very easily proved---see \cite{CCLLO24}.
But the cancellation condition \eqref{eq:weak-cancellation} actually corresponds to requiring that some first moments also vanish, and is stronger than necessary.

It is curious that, for functions $A$ supported in a tube $T(g,r,h)$, estimates of the form 
\[
\norm{ A }_{L^2(\Hn)} \lesssim r^{-2} \norm{ a }_{L^2(\Hn)} 
\]
hold when $a = \oper{L}_{\one} A$, but estimates of the form 
\[
\norm{ A }_{L^2(\Hn)} \lesssim r^{-1} \norm{ a }_{L^2(\Hn)} 
\]
do not hold when $a = \nabla_{\one} \cdot A$. 
If they did, then \eqref{eq:wrong-cancellation} would be a good cancellation condition for particles.

\section{Conclusion}
Hardy spaces have been applied to establish delicate estimates in PDE.
This little paper shows that an understanding of PDE (of a more geometric flavour) is needed to understand Hardy spaces.

\end{document}